\documentclass{elsarticle}
\usepackage{array}
\usepackage{tikz}

\usepackage{amssymb,latexsym,amsthm}
\usepackage{amsmath,amsfonts}
\usepackage{enumerate}

\usepackage{geometry}
\usepackage{url}
\ifdefined\directlua
\usepackage{fontspec}
\usepackage{polyglossia}
\setmainlanguage{english}
\usepackage{microtype}
\usepackage{pdftexcmds}
\makeatletter
\ifcase\pdf@shellescape
\relax\or
\begingroup
  \catcode`\%=12\relax
  \gdef\sformat{"Date: %ci Commit: %h"}
\endgroup
\directlua{
 local cmd="git show -s --format='"..\sformat.."'"
 local r=io.popen(cmd):read("*a")
 if (r) then
      tex.print("\string\\def\string\\COMMIT{"..r.."}")
 end
 }
\or
\relax\fi
\makeatother
\ifdefined\COMMIT
        \usepackage{background}
        \backgroundsetup{%
         pages=all, placement=bottom,angle=0,scale=2,%
         vshift=20pt,%
         contents={\COMMIT}}
\fi
\else
\usepackage[british]{babel}
\fi
\newcommand{\KK}{\mathbb{K}}
\newcommand{\RR}{\mathbb{R}}
\newcommand{\SSs}{\mathbb{S}}
\newcommand{\CC}{\mathbb{C}}
\newcommand{\cS}{\mathcal{S}}
\newcommand{\cD}{\mathcal{D}}
\newcommand{\cB}{\mathcal B}
\newcommand{\cA}{\mathcal A}
\newcommand{\cU}{\mathcal U}
\newcommand{\cQ}{\mathcal Q}

\newcommand{\wt}{\mathrm{wt}}
\newcommand{\cL}{\mathcal{L}}
\newcommand{\cH}{\mathcal{H}}
\newcommand{\GG}{\mathbb{G}}
\newcommand{\HH}{\mathbb{H}}
\newcommand{\PP}{\mathbb{P}}
\newcommand{\XX}{\mathbb{X}}
\newcommand{\fB}{\mathfrak{B}}
\newcommand{\fS}{\mathfrak{S}}
\newcommand{\fX}{\mathfrak{X}}
\newcommand{\fP}{\mathfrak{P}}
\newcommand{\fL}{\mathfrak{L}}
\newcommand{\fA}{\mathfrak{A}}
\newcommand{\fC}{\mathfrak{C}}
\newcommand{\Gr}{\mathrm{Gr}}
\newcommand{\cM}{\mathcal{M}}
\newcommand{\cC}{\mathcal{C}}
\newcommand{\cV}{\mathcal{V}}
\newcommand{\cW}{\mathcal{W}}
\newcommand{\cT}{\mathcal{T}}
\newcommand{\cX}{\mathfrak{X}}
\newcommand{\cF}{\mathcal{F}}
\newcommand{\cP}{\mathcal{P}}

\newcommand{\cN}{\mathcal{N}}
\newcommand{\fa}{\mathfrak a}
\newcommand{\sh}{\mathrm{sh}}
\newcommand{\fb}{\mathfrak b}
\newcommand{\fx}{\mathfrak x}
\newcommand{\fv}{\mathfrak v}
\newcommand{\diam}{\mathrm{diam}}
\newcommand{\PgL}{\mathrm{P}\Gamma{\mathrm{L}}}
\newcommand{\LL}{\mathbb{L}}
\newcommand{\TT}{\mathbb{T}}
\newcommand{\QQ}{\mathbb{Q}}
\newcommand{\cG}{\mathcal{G}}
\newcommand{\cGG}{\mathcal{Gr}}
\newcommand{\ccQ}{\mathcal{Q}}
\newcommand{\RM}{\mathrm{RM}\,}
\newcommand{\trace}{\mbox{\itshape trace}}
\newcommand{\diag}{\mbox{\itshape diag}}
\newcommand{\spin}{\text{\itshape spin}}
\newcommand{\cha}{\text{\itshape char}}
\newcommand{\bE}{\mathbb E}
\newcommand{\baM}{\overline{M}}
\newcommand{\PG}{\mathrm{PG}}
\newcommand{\Sp}{\mathrm{Sp}}
\newcommand{\GL}{\mathrm{GL}}
\newcommand{\PGL}{\mathrm{PGL}}
\newcommand{\PGO}{\mathrm{PGO}}
\newcommand{\PGU}{\mathrm{PGU}}
\newcommand{\PSp}{\mathrm{PSp}}
\newcommand{\FF}{\mathbb{F}}
\newcommand{\ZZ}{\mathbb{Z}}
\newcommand{\NN}{\mathbb{N}}
\newcommand{\WW}{\mathbb{W}}
\newcommand{\bS}{\mathbb{S}}
\newcommand{\Rad}{\mathrm{Rad}}
\newcommand{\Res}{\mathrm{Res}}
\newcommand{\Fix}{\mathrm{Fix}}
\newcommand{\Aut}{\mathrm{Aut}}
\newcommand{\lt}{\mathrm{lt}}
\newcommand{\gr}{\mathrm{gr}}
\newcommand{\er}{\mathrm{er}}
\newcommand{\aut}{\mathrm{Stab}}
\newcommand{\ch}{\mathrm{char}}
\newcommand{\rank}{\mathrm{rank}\,}
\newcommand{\N}{\mathcal{N}}
\newcommand{\ox}{\overline{x}}
\newcommand{\ov}{\overline{v}}
\newcommand{\oy}{\overline{y}}
\newcommand{\oU}{\widetilde{U}}
\newcommand{\oS}{\overline{S}}
\newcommand{\oM}{\overline{M}}
\newcommand{\ou}{\overline{u}}
\newcommand{\oV}{\overline{V}}
\newcommand{\oPi}{{\overline{\Pi}}_{\varphi}}
\newcommand{\oRad}{\overline{\mathrm{Rad}}(\varphi)}
\newcommand{\GF}{\mathrm{GF}}

\newcommand{\bZ}{\bf{0}}
\newcommand{\codim}{\mathrm{codim}\,}
\theoremstyle{plain}
\newtheorem{MainTheo}{Theorem}

\newtheorem{lemma}{Lemma}[section]
\newtheorem{theorem}[lemma]{Theorem}
\newtheorem{corollary}[lemma]{Corollary}

\newtheorem{prop}[lemma]{Proposition}
\newtheorem{claim}{Claim}

\theoremstyle{definition}
\newtheorem{definition}[lemma]{Definition}
\newtheorem{remark}[lemma]{Remark}

\def\pr{\noindent{\bf Proof. }}
\def\eop{\hspace*{\fill}$\Box$}
\begin{document}
%\begin{frontmatter}

\title{On the Grassmann Graph of Linear Codes}
	\author[1]{Ilaria Cardinali\corref{cor1}}
\ead{ilaria.cardinali@unisi.it}
\author[2]{Luca Giuzzi}
  \ead{luca.giuzzi@unibs.it}
\author[3]{Mariusz Kwiatkowski}
\ead{mkw@matman.uwm.edu.pl}
\address[1]{
  Department of Information Engineering and Mathematics,
University of Siena,
Via Roma 56, I-53100, Siena, Italy}
\address[2]{D.I.C.A.T.A.M. (Section of Mathematics),
University of Brescia,
Via Branze 53, I-25123, Brescia, Italy}
\address[3]{Faculty of Mathematics and Computer Science,
University of Warmia and Mazury, S{\l}oneczna 54, Olsztyn, Poland}
\cortext[cor1]{Corresponding author}

\begin{abstract}
Let $\Gamma(n,k)$ be the Grassmann graph formed by the $k$-dimensional subspaces of a vector space of dimension $n$ over a field $\FF$ and, for $t\in \mathbb{N}\setminus \{0\}$, let $\Delta_t(n,k)$ be the subgraph  of $\Gamma(n,k)$ formed by the set of linear $[n,k]$-codes having minimum dual distance at least $t+1$.
We show that if $|\FF|\geq{n\choose t}$ then $\Delta_t(n,k)$ is connected and
  it is isometrically embedded in $\Gamma(n,k)$. %This generalizes some results of \cite{KPP18} and \cite{KP16}.
\end{abstract}

\maketitle
\par

\section{Introduction}\label{Introduction}
Let $V:=V(n,\FF)$ be a $n$-dimensional  vector space
over a field $\FF$ and for $k=1,\dots, n-1$, denote by
$\Gamma(n,k)$ the \emph{$k$-Grassmann graph} of $V$, that is the graph
whose vertices are the $k$-subspaces of $V$ and where two
vertices $X,Y$ are connected by an edge if and only if
$\dim(X\cap Y)=k-1$. See~\cite{DRG} for more detail.

It is interesting to see what properties extend from the
graph $\Gamma(n,k)$ to some of its subgraphs.

Suppose that $B=(e_1,\dots,e_n)$ is a given basis of $V$; henceforth we will write the coordinates of the vectors in $V$ with respect to $B$.
Given two vectors $x=\sum_{i=1}^n{\alpha_i}e_i$ and $y=\sum_{i=1}^n{\beta_i}e_i$,
the Hamming distance (with respect to the basis $B$) between $x$ and $y$
is $d(x,y):=|\{i: x_i\neq y_i\}|.$
In this setting, a \emph{$[n,k]$-linear code} $C$ is just
a $k$-dimensional vector subspace of $V.$
Usually it is also assumed that $V$ is defined over
a finite field $\FF_q$. However, for the
purposes of the present paper we shall use the language of coding
theory even when the field $\FF$ is not finite.
If $B_C$  is an ordered  basis of $C$,  a \emph{generator matrix}
for $C$ is the $k\times n$ matrix whose rows are the coordinates of the elements of $B_C$ with respect to $B$.
Given a $[n,k]$-linear code $C$, its \emph{dual code}
is the $[n,n-k]$-linear code $C^{\perp}$ given by
\[ C^{\perp}:=\{ v\in V: \forall c\in C, v\cdot c=0 \} \]
where by $\cdot$ we mean the standard symmetric bilinear form on $V$
given by
\[ (v_1,\dots,v_n)\cdot (c_1,\dots,c_n)=v_1c_1+\dots+v_nc_n. \]
Since the  $\cdot$ is non-degenerate, $C^{\perp\perp}=C$.
We say that $C$ has \emph{dual minimum distance} $t+1$ if
and only if the minimum Hamming distance of the dual $C^{\perp}$ of $C$ is $t+1$.

%;
%this is the same as to say that
%\[ \forall c'\in C^{\perp}, w(c'):=|\{i:c_i\neq 0\}|\geq t+1 . \]
The condition for a $[n,k]$-linear code $C$ having dual minimum distance at least $t+1$ can be easily read on any  generator matrix of it.
Indeed (see, e.g., Proposition~\ref{equiv}), $C$ has dual minimum distance at least $t+1$ if and only if any $t$-columns of any generator matrix of $C$ are linearly independent.

For $t\in \mathbb{N}\setminus \{0\}$, let $\cC_t(n,k)$ be the set of all $[n,k]$-linear codes with dual minimum distance at least $t+1$ and denote by $\Delta_t(n,k)$ the subgraph of $\Gamma(n,k)$ induced by the elements of $\cC_t(n,k)$, i.e. the vertex set of $\Delta_t(n,k)$ is formed by the elements in $\cC_t(n,k)$ and two vertices $X$ and $Y$ are adjacent in $\Delta_t(n,k)$ if and only if $\dim(X\cap Y)=k-1.$ We shall call $\Delta_t(n,k)$ the \emph{Grassmann graph} of the linear $\cC_t(n,k)$ codes.

Note that for $t=1$, $\cC_1(n,k)$ is the class of the \emph{non-degenerate $[n,k]$-linear codes} and for $t=2$, $\cC_2(n,k)$ is the class of the \emph{ projective $[n,k]$-linear codes}.

In general,  we say that a subgraph is isometrically embedded in a larger graph if there exists a distance-preserving map among them (see also Definition~\ref{def:isometrically}); see also~\cite{S}.

In~\cite{KP16} Kwiatkowski and Pankov studied the graph
$\Delta_1(n,k)$ and more recently in~\cite{KPP18} Kwiatkowski, Pankov and Pasini,
considered the graph $\Delta_2(n,k)$ in the case $\FF$ is a finite field of order $q$.

In~\cite[Corollary 2]{KP16}, the authors show that $\Delta_1(n,k)$
is connected and isometrically embedded in $\Gamma(n,k)$ if
and only if $n<(q+1)^2+k-2$.
In~\cite[Theorem 1]{KPP18} it is shown that a sufficient condition
for the graph $\Delta_2(n,k)$ to be isometrically embedded in $\Gamma(n,k)$
is $q\geq{n\choose 2}$.  In~\cite{KPP18} it is also shown that the
graph of simplex codes is not always isometrically embedded in the
Grassmann graph.

In this paper we extend some results of~\cite{KP16} and~\cite{KPP18} to the graphs $\Delta_t(n,k)$ for arbitrary $t\leq k$ and arbitrary fields $\FF$.

More in detail, our main result is
the following
\begin{MainTheo}  \label{c1}
  Let $t,k,n$ be integers such that  $1\leq t\leq k\leq n<\infty$.
  Suppose that $\FF$ is a field  with $|\FF|\geq{n\choose t}$. Then
  the graph $\Delta_t(n,k)$ is connected and isometrically embedded into the $k$-Grassmann graph $\Gamma(n,k)$.
  Furthermore, the diameter of $\Delta_t(n,k)$ and $\Gamma(n,k)$
  are the same.
\end{MainTheo}

\begin{remark}
  The hypotheses in Theorem~\ref{c1} are sufficient for the
  graph $\Delta_t(n,k)$ to be connected and to be isometrically
  embedded into $\Gamma(n,k)$ but in general they are not necessary;
  see also~\cite[Corollary 2]{KPP18}.
  We leave to a future work to determine
  if the graph $\Delta_t(n,k)$ might be connected also under
  some weaker assumptions on $q$ or $t$ and see if there are cases where
  the embedding is not isometric. We leave also to a future
  work to generalize these results to vector spaces with Hamming distance
  over a possibly non-commutative division ring as well as
  to the infinite dimensional cases both for $n$ and for $k$.
\end{remark}
%\begin{remark}
%Most of our results hold also if we work with arbitrary infinite fields rather than finite fields.
%Anyhow we have preferred to state all the results sticking to finite fields because otherwise the concept of linear codes itself would have required to be explained carefully and we feel this is not the right place to do it. Said this, we will point out here and there when the stated result holds indeed also for the case of infinite fields.
%  In general, for $q=\infty$ the notations $\cC_t(n,k;\infty)$ or,
%  equivalently $\Delta_t(n,k;\infty)$ are not well defined.
%  However, since our results depend actually only on the cardinality
%  of the field $\FF_q$ and not on further properties of the
%  Grassmannian, we feel there is no need to use
%  a more complex (but also more precise) symbol here.
%\end{remark}

The paper is structured as follows.
In Section~\ref{Intro} we recall some basic definitions and preliminary
results which shall be used in order to prove Theorem~\ref{c1}.
Section~\ref{mres} contains the proof of our main results; in particular,
in Subsection~\ref{conn} we shall prove that the graph $\Delta_t(n,k)$ is
connected and isometrically embedded in $\Gamma(n,k)$ for $q\geq{n\choose t}$,
while in Subsection~\ref{maxd} we shall show that the sets
$\cC_t(n,k)$, when not empty, always contain codes which are at maximum distance in
the Grassmann graph $\Gamma(n,k)$. %Finally, in Subsection~\ref{pfmain}
%we prove Theorem~\ref{c1}.

\section{Preliminaries}\label{Intro}
As mentioned in the Introduction, $\FF$ is a field and $V:=V(n, \FF)$ denotes a $n$-dimensional vector space over $\FF$. Let  $B=(e_1,\dots, e_n)$ be a given ordered basis of $V$ with respect to which all the vectors will be written in coordinates.
For $k$ and $t$ integers such that $1\leq t\leq k\leq n-1$, $\cC_t(n,k)$ is the class of $[n,k]$-linear codes having dual minimum distance at least $t+1$. More explicitly,
    \[ \cC_t(n,k):=\{ C\subseteq V : \dim C=k, d^{\perp}(C)\geq t+1\} \]
where $d^{\perp}(C):=d_{\min}(C^{\perp})$ is the minimum distance of the dual code $C^\perp$ which means that the weight $wt(v)$ of any codeword $v=(v_1,\dots, v_n)\in C^\perp$ is at least $t+1$, i.e.
\[\wt(v):=|\{i:v_i\neq 0\}|\geq t+1.\]

 Clearly, if $\cC_t(n,k)\neq\emptyset$, then necessarily $t\leq k\leq n$.

     If $t=k$ and $\FF=\FF_q$, the elements of $\cC_k(n,k)$ are exactly the \emph{maximum distance
      separable} $[n,k]$-codes (see e.g.~\cite{MS}),
    that is codes whose minimum distance $d_{\min}$
    attains the Singleton bound $d_{\min}=n-k+1$; see Corollary~\ref{t=1,2}.

   \begin{remark}
    The condition $t\leq k\leq n$ is necessary but
    in general not sufficient to ensure
    that $\cC_t(n,k)$ is not empty. Indeed, if $\FF=\FF_q$, even for
    arbitrary values of $q$, it is not straightforward to
    determine if $\cC_t(n,k)\neq\emptyset$ or characterize
    the elements of $\cC_t(n,k)$.
    For instance the celebrated MDS conjecture implies
    $\cC_k(n,k)=\emptyset$ for $n>q+2$.
    On the other hand, if $n<q+1$,  then by Lemma~\ref{mds},
    $\cC_t(n,k)\neq\emptyset$ for all $t\leq k\leq n$.
  \end{remark}
  In order to avoid trivial cases, we shall henceforth  suppose that the parameters $n,k,t$ and $q$ if $\FF:=\FF_q$, have been chosen so
  that $\cC_t(n,k)\neq\emptyset$; in Lemma~\ref{mds} it shall
  be shown that under the assumptions of Theorem~\ref{c1} this is
  always true.
%\begin{definition}
%  We denote by $\Gamma(n,k)$ the  \emph{$k$-Grassmann graph} of $V$, that
%  is the graph whose vertices are the $k$-dimensional subspaces of $V$
%  and where there is an edge between two subspaces $X$ and $Y$
%  if and only if $\dim(X\cap Y)=k-1$.

 We recall from the Introduction that $\Delta_t(n,k)$ is  the subgraph of $\Gamma(n,k)$ induced
  by the elements of $\cC_t(n,k)$.
  %More in detail, a $k$-subspace $X$ of $V$ is a vertex of $\Delta_t(n,k)$
  %if and only if $X\in\cC_t(n,k)$ and any two vertices $X,Y\in\cC_t(n,k)$
  %are adjacent if and only if $\dim(X\cap Y)=k-1$.
%\end{definition}
\begin{definition}\label{def connection}
  Let $X\in\Delta_t(n,k)$. We define the
  \emph{connected component} $\Delta_t^X(n,k)$ of $X$ in $\Delta_t(n,k)$  as the subgraph of $\Delta_t(n,k)$ whose vertices are all
  $Y\in\Delta_t(n,k)$ such that there is a path in
  $\Delta_t(n,k)$ joining $X$ and $Y$.
  The graph $\Delta_t(n,k)$ is \emph{connected}
  if $\Delta_t^X(n,k)=\Delta_t(n,k)$ for some (and, consequently
  for all) $X\in\Delta_t(n,k)$.
\end{definition}

For any $X,Y\in\cC_t(n,k)$ write $d(X,Y)$ for the distance between $X$ and $Y$ in the Grassmann graph $\Gamma(n,k)$ and
$d_t(X,Y)$ for the distance between $X$ and $Y$ in $\Delta_t(n,k).$
If $\Delta_t^X(n,k)\neq\Delta_t^Y(n,k)$, that is
$X$ and $Y$ are in different connected components of $\Delta_t(n,k)$ we put
$d_t(X,Y)=\infty$.
We recall that the {\it diameter} of a graph
is the maximum of the distances among two of its vertices.
%\[ \mathrm{diam}(\Gamma):=\max_{X,Y\in E} d(X,Y). \]

Since every edge of $\Delta_t(n,k)$ is an edge
of $\Gamma(n,k)$, it is straightforward to see that $d_t(X,Y)\geq d(X,Y)$
for all $X,Y\in\Delta_t(n,k)$.

\begin{definition}
  \label{def:isometrically}
We say that $\Delta_t(n,k)$ is \emph{isometrically embedded} in
$\Gamma(n,k)$ if for any $X,Y\in\Delta_t(n,k)$ we have
$d_t(X,Y)=d(X,Y)=k-\dim(X\cap Y)$.
\end{definition}
% Isometric embeddings of graphs play an important role in the
% study of incidence geometries; see~\cite{S}.

Note that if $\Delta_t(n,k)$ is isometrically
embedded in $\Gamma(n,k)$, then $\Delta_t(n,k)$ is
also connected.

%For details on the graphs $\Delta_1(n,k)$ and $\Delta_2(n,k)$ for
%$q<\infty$ we refer respectively to~\cite{KP16} and~\cite{KPP18}. So, the
%present paper we shall focus mostly on the cases $t\geq3$ even
%if our arguments can be applied also for $t<3$.

\subsection{Some basic results}

\begin{definition}\label{coord subspace}
Let $(i_1,\dots,i_t)\in\NN^t$ be a $t$-uple of integers such
that $1\leq i_1<i_2<\dots < i_t\leq n$.
We denote by $C_{i_1\dots i_t}:=\cap_{j:=1}^t (x_{i_j}=0)$ the
$(n-t)$-dimensional subspace of $V$ obtained as the intersection of the coordinate hyperplanes of $V$ of equations
$x_{i_j}=0.$ We shall call $C_{i_1\dots i_t}$ the
$(i_1,\dots,i_t)$-\emph{coordinate subspace} of $V.$
\end{definition}
%\begin{definition}
%Given a $[n,k]$-code $X$,
%\emph{generator matrix} $G_X$ for $X$ is a $k\times n$ matrix
%whose rows contain the component of a basis $(b_1,b_2,\dots,b_k)$
%of $X$ with respect
%to the canonical basis of $\FF_q^n$.
%\end{definition}

The \emph{monomial group} $\cM(V)$ of $V$ consists of all linear
transformations of $V$ which map the set of subspaces
$\{ \langle e_1\rangle,\dots,\langle e_n\rangle \}$ in itself.
It is straightforward to see that
${\cM(V)}\cong \FF^*\wr S_n$ where $\wr$ denotes the wreath product  and $S_n$ is the symmetric group of order $n$; see~\cite[Chapter 8,\S5]{MS} for more details.

\begin{definition}\label{def equiv}
  Two $[n,k]$-linear codes $X$ and $Y$ are \emph{equivalent}
  if there exists a monomial transformation $\rho\in\cM(V)$ such
  that $X=\rho(Y)$.
\end{definition}

Suppose $X$ is a $[n,k]$-linear code  with generator matrix $G_X$. If $A\in\GL(k,\FF)$ then $G'_X=AG_X$ is also a generator matrix for $X$.

It follows that
two $[n,k]$-linear codes $X$ and $Y$ with generator matrices
respectively $G_X$ and $G_Y$ are \emph{equivalent} if there
exists $A\in\GL(k,\FF)$, a permutation matrix $P\in\GL(n,\FF)$
and a diagonal matrix $D\in\GL(n,\FF)$ such that
\[ G_X=AG_Y(PD). \]

Equivalence between linear codes is an equivalence relation and
the equivalence class of a code $X$ corresponds to the orbit of
$X$ under the action of $\cM(V)$ on the $k$-dimensional subspaces
of $V$.

Also, it can be readily seen that
two codes are equivalent if and only if any two of
their generator matrices belong to the same orbit under
the action of the group $\PGL(k,\FF):(\FF^*\wr S_n)$, where $\PGL(k,\FF)$ acts on the right
of the generator matrix.

With mostly harmless abuse of notation,
in the remainder of this paper we shall not distinguish between the action of
$\cM(V)$ on the codes (regarded as subspaces of $V$) and that on
the columns of their generator matrices.

Since equivalent codes have the same parameters (in particular they have
the same minimum dual distance), we have that $\cC_t(n,k)$
consists of unions of orbits under the action of ${\cM(V)}$.

%\begin{remark}
%  We leave to a future work to describe, as the parameter
%  vary, what is the subgroup of $\PGL(k,q):(\FF_q^*\wr S_n)$
%  acting faithfully on the generator matrices of the codes
%  in $\cC_t(n,k)$.
%\end{remark}

For $j=1,\dots, n,$ let  $x^j\colon V \rightarrow \FF_q$ be the $j^{\mathrm{th}}$-coordinate linear functional of $V$ which acts on the vectors $e_i$, $1\leq i\leq n$, of $B$  as $x^j(e_i)=\delta_{ij}$, where
by $\delta_{ij}$ we mean the Kronecker $\delta$ function.

Observe that, for any $v\in V$ and $j$ with $1\leq j\leq n$, we have that
$x^j(v)$ is exactly the $j$-th component of $v$
with respect to the basis $B$. So, if $X$ is a $[n,k]$-linear code and $B_X=(b_1,\dots, b_k)$ is a given basis of $X$ with respect to which the generator matrix $G_X$ is written, then  for any $i,j$ with $1\leq i\leq k$ and  $1\leq j\leq n,$ then $x^j(b_i)$ is exactly the $(i,j)$- entry in the matrix  $G_X$. So, the $j$-th column of $G_X$ represents the restriction $x^j|_{B_X}$ of the functional
$x^j$ to the basis $B_X$. By linearity, we can say that the $j$-th column of $G_X$ represents the restriction $x^j|_X$ of the functional
$x^j$ to $X$.
This has the following important consequence.
\begin{lemma}
  \label{lcons}
  Let $X\subseteq V$ be a $[n,k]$-linear code and $G_X$ be
  a generator matrix of $X$.
  A set of coordinate functionals  restricted to $X$ is linearly independent
  if and only if the columns of $G_X$ representing them are linearly
  independent.
\end{lemma}

If $\FF:=\FF_q$, we shall also use the notation
\[ [m]_q:=\frac{q^{m}-1}{q-1} \]
for the number of $1$-dimensional subspaces
of an $m$-dimensional vector space.

%In particular $[m]_q$ is
%the number of points of a projective space $\PG(m-1,q)$.
%If $q=\infty$, we use the convention
%$[1]_{\infty}=1$ and $[m]_{\infty}=\infty$ for
%all $m>0$.
% def distance d_t

The equivalence between~(\ref{cc1}) and~(\ref{cc2}) in the following
proposition is well known; however, since many results  of the present work rely on it, we present a complete proof for the convenience of the reader.
\begin{prop}\label{equiv}
Let $X$ be a $[n,k]$-linear code and denote by $G_X$ a generator matrix of $X$. The following are equivalent.
\begin{enumerate}[(1)]
\item\label{cc1} $X$ has minimum dual distance at least $t+1$.
\item\label{cc2} Any $t$ columns of $G_X$ are linearly independent.
\item\label{cc3}  For any $1\leq i_1<i_2<\dots < i_t\leq n$ we have $dim(X\cap C_{i_1\dots i_t})=k-t$ where $C_{i_1\dots i_t}:=\cap_{j:=1}^t (x_{i_j}=0)$ is the $(n-t)$-dimensional $(i_1,\dots, i_t)$-coordinate subspace of $V.$
\end{enumerate}
\end{prop}
\begin{proof}
  The matrix $G_X$ is a parity check matrix for the dual code $X^{\perp}$. Write the columns of $G_X$ as $G_1,\dots,G_n$ and let
  $y=(y_1,\dots,y_n)\in X^\perp$.  Then
  \begin{equation}\label{me}
    G_Xy^t=G_1y_1+\dots+G_ny_n=\mathbf{0}.
  \end{equation}
  Assume~(\ref{cc1}). Then, for any $y\in X^{\perp}$, $y\neq\mathbf{0},$ we
  have $\mathrm{wt}(y)\geq t+1$.
  Suppose by contradiction that there is a set of $t$-columns of
  $G_X$ which are linearly dependent. To simplify the exposition, assume without much loss
  of generality that this set comprises the first $t$-columns.
  Then,
  \[ G_1y_1+G_2y_2+\dots+G_ty_t=\mathbf{0} \] with at least one entry $y_i$ different from $0$; so
  the vector $(y_1,\dots,y_t,0,\dots,0)\neq\mathbf{0}$ is in $X^{\perp}$
  with $\mathrm{wt}(y)\leq t<t+1$. This contradicts~(\ref{cc1}).

  Conversely, assume~(\ref{cc2}) and take $y\in X^{\perp}$ with
  $\mathrm{wt}(y)=d$.
  Then $G_Xy^T=0$. If $d=0$, that is $y=\mathbf{0}$, then
  there is nothing to prove.
  If $d\not= 0$, suppose, again without much loss of generality, that exactly the first $d$ entries $y_1,\dots, y_d$
  of $y$ are non-zero. Then
  \[ G_1 y_1+\dots+G_d y_d=\mathbf{0}. \]
  In particular, the first $d$ columns of $G_X$ must be
  linearly dependent; by~(\ref{cc2}) we necessarily have $d>t$ since
  any set of $t$ columns of $G_X$ is independent; this implies
  (\ref{cc1}).
  \par
  We now prove the equivalence between~(\ref{cc2}) and~(\ref{cc3}).
  Suppose that~(\ref{cc3}) holds. Then,
  $\dim(X\cap C_{i_1\dots i_t})=k-t$ which means that the restrictions $x^{i_1}|_X, \dots, x^{i_t}|_X$ of the
  $t$ coordinate functionals
  $x^{i_1},\dots,x^{i_t}$   of $V$
  to $X$ are linearly independent. Then~(\ref{cc2})
  follows from Lemma~\ref{lcons}.
%  In particular, they must give linearly independent vectors
%  when evaluated on the basis of $X$ given by the rows of $G_X$;
%  this implies that the corresponding columns $G_{i_1},\dots,G_{i_t}$ of
%  $G_X$  must be independent; so~\ref{cc2}
%  follows.

  Conversely, assume~(\ref{cc2}) holds and suppose by contradiction
  that~(\ref{cc3}) is false, that is that there exists a set of indexes
  $i_1,\dots,i_{t}$ such that $\dim(X\cap C_{i_1\dots i_t})\geq k-t+1$.
  Then, for some $j\in \{1,\dots, t\},$ we have
  \[ X\cap C_{i_1\dots i_{j-1}i_{j+1}\dots i_{t}}\subseteq
    X\cap C_{i_j}. \]
  In terms of coordinate functionals this means
  \[ x^{i_j}|_X \in \langle x^{i_1}|_X,\dots,
    x^{i_{j-1}}|_X,x^{i_{j+1}}|_X,\dots, x^{i_t}|_X\rangle. \]
  So $x^{i_j}|_X$ is a linear combination of the remaining coordinate
  functionals. In particular, by Lemma~\ref{lcons}, this means that the column $G_{i_j}$ of
  any generator matrix $G_X$ of $X$ is a linear combination of the
  columns $G_{i_1},\dots, G_{i_{j-1}},G_{i_{j+1}},\dots,G_{i_t}$. This contradicts~(\ref{cc2}).
\end{proof}

The following is an immediate consequence of Proposition~\ref{lcons}.
\begin{corollary}
  \label{t=1,2}
  The set  $\cC_1(n,k)$ consists of all $[n,k]$-linear non-degenerate
  codes; $\cC_2(n,k)$ consists of all $[n,k]$-linear projective
  codes; the set $\cC_k(n,k)$, if $\FF=\FF_q$, consists of all $[n,k]$-linear
  MDS codes.
\end{corollary}
\begin{proof}
  Only the statement about $\cC_k(n,k)$ needs to be proved as
  the descriptions of $\cC_1(n,k)$ and $\cC_2(n,k)$ follow
  directly from Proposition~\ref{equiv}.
  Suppose $C\in \cC_k(n,k)$, i.e. $C$ is  a $[n,k]$-code having dual minimum distance at least $k+1$.
  Then, by definition of $\cC_k(n,k)$,  $C^{\perp}$ is a $[n,n-k]$-code with minimum
  distance at least $k+1=n-(n-k)+1$ and, as such it is a MDS-code.
  Since the duals of MDS codes are MDS codes, $C=C^{\perp \perp}$ is also MDS.
  \par
  Conversely, suppose $C$ to be a $[n,k]$-linear
  MDS code; then $C^{\perp}$ is also MDS
  and has minimum  distance $k+1$. It follows that $C\in\cC_k(n,k)$.
\end{proof}

\section{Proof of Theorem~\ref{c1}}\label{mres}
We proceed by steps. First, we show that $\cC_t(n,k)$ is not empty for all $t$ and $k$ with
  $1\leq t\leq k\leq n$ under the hypothesis that $\FF$ is a field with $|\FF|+1\geq n.$
Then, in Section~\ref{conn} we provide
a condition for the graph $\Delta_t(n,k)$ to be connected and
isometrically embedded in $\Gamma(n,k)$ and in Section~\ref{maxd} we show that any class $\cC_t(n,k)$
contains elements which are at maximum distance in $\Gamma(n,k).$
\begin{lemma}
  \label{mds}
  If $\FF$ is a field with $|\FF|+1\geq n$ then $\cC_t(n,k)\neq\emptyset$ for all $t$ and $k$ with
  $1\leq t\leq k\leq n$.
\end{lemma}
\begin{proof}
Note that for all $t$ we have $\cC_{t}(n,k)\subseteq\cC_{t-1}(n,k)$.
  So, in order to get the lemma we just need to show that
  $\cC_k(n,k)\neq\emptyset$ under our assumptions.
  It is well known that if $n\leq q+1$ for $\FF:=\FF_q$  a finite field or order $q$, there exist
  $[n,k]$--linear MDS codes. Since the dual of an MDS code
  is MDS, it is immediate to see that any $k$-columns of the
  generator matrix of a $[n,k]$--MDS code $C$ are independent.
  It follows that $C\in\cC_{k}(n,k)\neq\emptyset$.

  Suppose now that $\FF$ is an arbitrary infinite field. There exist at least $n$ distinct
  elements $a_1,a_2,\dots,a_n\in\FF$.
  Consider the matrix
  \[ G:=\begin{pmatrix}
      1     & 1     & \dots & 1 \\
      a_1   & a_2   & \dots & a_n  \\
      a_1^2 & a_2^2 & \dots & a_n^2 \\
      \vdots & \vdots & & \vdots \\
      a_1^{k-1} & a_2^{k-1} & \dots & a_n^{k-1}
    \end{pmatrix}. \]
  Any $k\times k$ minor $M_{i_1\dots i_k}$ of $G$, comprising the columns
  $i_1,\dots,i_k$ is a Vandermonde matrix with determinant
  \[ \det(M_{i_1\dots i_k})=\prod_{1\leq r<s\leq k} (a_{i_s}-a_{i_r})\neq 0. \]
    In particular, the code $C$ with generator matrix $G$ belongs
    to $\cC_k(n,k)$ which is consequently non-empty.
\end{proof}

\subsection{The connectedness of the graph}
\label{conn}

The following are two elementary lemmas of linear algebra.
\begin{lemma}  \label{prelim}
  Let $X\in \cC_t(n,k)$ and let $H$ be a hyperplane of $X.$
  If  $y\not\in X$ then
  \[ \dim(\langle H, y\rangle\cap C)\leq k-t+1 \]
  for every $(n-t)$-dimensional coordinate subspace $C$ of $V.$
\end{lemma}
\begin{proof}
 Suppose by contradiction
   \[ \dim(\langle H, y\rangle\cap C)\geq k-t+2. \]
   Since $H\subseteq X$ we also have
   $\dim(\langle X, y\rangle\cap C)\geq k-t+2$.

   Any vector
   in $\langle H, y\rangle$ can be
   written in the form $x+\alpha y$ where $x\in H$ and $\alpha\in \FF$.
   In particular, there are $k-t+2$ linearly independent vectors $v_i$ in $\langle X, y\rangle\cap C$
   of the form
   \[ v_i=x_i+\alpha_i y \]
where $x_i\in H$ and $\alpha_i\in \FF.$
   By Gaussian elimination (we remove $y$), we have at least $k-t+1$ vectors
   in $X$ which are linearly independent and contained in $C$;
   so $\dim (X\cap C)\geq k-t+1$, which is a contradiction because
   $X\in \cC_t(n,k)$ (see Proposition~\ref{equiv}).
\end{proof}

\begin{lemma}
  \label{trivial}
  Let $S$ be a vector space of dimension $s$,
  $H_1\neq H_2$ be two distinct hyperplanes
  of $S$ with fixed bases $B_1$ and $B_2$.
  Then, there exists a basis $B$ of $S$ contained in $B_1\cup B_2$.
\end{lemma}
\begin{proof}
    Since $H_1\neq H_2$, there exists at least one element $b\in B_2\setminus H_1$.
  Consider $B=B_1\cup\{b\}$. This is a linearly independent set consisting
  of $s$ distinct elements, $B\subseteq S$ and $\dim(S)=s$. It follows
  that $B$ is a basis of $S$ with $B\subseteq B_1\cup B_2$.
\end{proof}

Recall that by $d(X,Y)$ we mean the distance in $\Gamma(n,k)$ while $d_t(X,Y)$ denotes the distance in $\Delta_t(n,k)$.

The following definitions are used in the proof of Lemma~\ref{lcmm1}.
\begin{definition}  \label{psi-def}
Put ${{\{1,\dots,n\}}\choose t}:=\{(i_1,\dots,i_t)\in \mathbb{N}^t\colon 1\leq i_1<i_2<\dots<i_t\leq n\}$ and define as
  \emph{colors} the elements of it. Endow ${{\{1,\dots,n\}}\choose t}$ with the natural lexicographic order on the
  $t$-uples.

    Take $X,Y\in\cC_t(n,k)$ with $X\not= Y$ and let $H$ be a hyperplane of $X$
    such that $X\cap Y\subseteq H$. The
    \emph{coloration induced by $H$} is the
    map
    \[ \psi_H: Y/(H\cap Y)\to {{\{1,\dots,n\}} \choose t}\cup\{
      \mathbf{\infty}\} \]
  sending any vector $[p]\in Y/(H\cap Y)$ to
  the smallest (in the lexicographic order)
  color $(i_1,\dots,i_t)$ such that
  \[ \dim(\langle H,p\rangle\cap C_{i_1\dots i_t}))= k-t+1 \]
  where $C_{i_1\dots i_t}$ is the
  $(i_1,\dots ,i_t)$-coordinate subspace as defined
  in Definition~\ref{coord subspace}. If no such color exists
  we put $\psi_H([p])=\mathbf{\infty}$.
\end{definition}

The function $\psi_H$ is well defined. Indeed, if $b\in [a]=a+(H\cap Y)$, then $b=a+h$ for some $h\in H\cap Y$ and
$\langle H,b\rangle=\langle H,a+h\rangle=\langle H,a\rangle$; so
$\psi_H([a])=\psi_H([b])$.

Henceforth we shall silently denote each element $[p]$ of $Y/(X\cap Y)$
by means of its representative element $p$.

\begin{definition}
  \label{ccca}
  Under the same assumptions as in Definition~\ref{psi-def},
  we say that a set $T\subseteq Y/(H\cap Y)$ is
  \emph{monochromatic} if $\forall r,s\in T$,
  $\psi_H(r)=\psi_H(s)\neq\infty$, i.e.
    all of its elements have the same color.
\end{definition}
\begin{definition}
  Under the same assumptions as in Definition~\ref{psi-def},
  we say that a subspace $S$ of
  $Y/(H\cap Y)$ with $\dim(S)=s$ is
  \emph{colorable} if there exists at least one
  monochromatic basis of $S.$
  If $S$ is colorable, we define the \emph{color} $\psi_H(S)$ of $S$
  as the minimum color of a basis of $S$.
  \end{definition}
  In symbols, let
    \[{\mathfrak F}(S):=\{f=(p_1, \dots, p_s)\colon  f  \mbox{ is a basis of } S \mbox{ and }\psi_H(p_1)=\dots =\psi_H(p_s)\neq\infty \}\]
    be the set of monochromatic bases of $S$.
    If $f\in {\mathfrak F}(S)$, denote by  $\psi_H(f)$  the color of any element in $f.$    Hence
  \begin{lemma}\label{equiv colorable}
  The subspace $S$ is colorable if and only if   ${\mathfrak F}(S)\neq\emptyset.$
  \end{lemma}

If $S$ is colorable, the color of $S$ is
    \[\psi_H(S):=\min \{\psi_H(f)\colon f\in {\mathfrak F}(S)\}.\]
    In other words,    $S$ has color $c$ if there are $s$ independent vectors  in
    $S$ all with the same color $c$ and any other set
  of $s$ independent vectors  in $S$ either is not monochromatic or
  has color $c'\geq c.$

  Note that a colorable subspace $S$ with color $c$ is not, in general,
  a monochromatic set.

%The following lemma holds regardless the finiteness of the cardinality of the field $\FF$ hence in its hypothesis we assume that $\FF$ can be either a finite field $\FF_q$ or an arbitrary infinite field. In the latter case we put by convention $\FF=\FF_q$ with $q:=\infty.$
\begin{lemma}
  \label{ml7}
  Let $X,Y\in\cC_t(n,k)$ with $\dim(X\cap Y)=k-d\geq k-t.$
  If $\FF$ is a field with $|\FF|\geq{n\choose t}$ then there exists a code
  $Z\in\cC_t(n,k)$ such that $\dim(X\cap Z)=k-1$ and $\dim(Z\cap Y)=k-d+1$.
\end{lemma}
\begin{proof}
  We prove that for every hyperplane $H$ of $X$ containing $X\cap Y$, there exists $z\in Y\setminus (X\cap Y)$ such that $Z:=\langle H,z\rangle\in \cC_t(n,k).$ %Note that  $\dim(Z)=k$, $d(Z,X)=1$ and $d(Z,Y)=d-1.$

 By way of contradiction suppose the contrary. Hence there exists a  hyperplane $H$ of $X$ with $X\cap Y\subseteq H$ such that for every  $z\in Y\setminus (X\cap Y)$,  we have  $\langle H,z \rangle \not\in \cC_t(n,k).$

Equivalently, by Proposition~\ref{equiv}, we suppose that there exists a  hyperplane $H$ of $X$ with $X\cap Y\subseteq H$ such that for every  $z\in Y\setminus (X\cap Y)$ there exist indexes $i_1<i_2<\dots<i_t$
such that $\dim(\langle  H,z\rangle\cap C_{i_1\dots i_t})\geq k-t+1.$
By Lemma~\ref{prelim}, we have $\dim(\langle  H,z\rangle\cap C_{i_1\dots i_t})= k-t+1.$

Under these assumptions, we will prove the following claim which leads to a contradiction.
%We denote by $[p_i]$ the equivalence class of the element
%$p_i$ in the quotient $Y/(X\cap Y)$.
\begin{claim}\label{claim}
There exist indexes $i_1,\dots,i_t$ and linearly independent
vectors $p_1,\dots,p_d\in Y\setminus(X\cap Y)$ such that
$[p_1],\dots,[p_d]$ are linearly independent in $Y/(X\cap Y)$,
\[\dim R_i= k-t+1\,\,{\rm and}\,\,\,\dim (R_1+\dots+R_d)\geq k-t+d\]
  where we put $H_i:=\langle H,p_i\rangle$ and $R_i=H_i\cap C_{i_1\dots i_t}.$
    %with $\dim R_i= k-t+1$ and $\dim (R_1+\dots+R_d)\geq k-t+d$.
\end{claim}

Note that for every $i$, $R_i\subseteq \langle H,Y\rangle.$
From Claim~\ref{claim}, we have
\begin{equation}
  \dim(Y+R_1+\dots+R_d)\leq\dim(H+Y)= k+d-1
\end{equation}
and
\begin{multline}
  \dim (Y\cap (R_1+\dots+R_d))=\dim(Y)+\dim(R_1+\dots+R_d)-
    \dim(R_1+\dots+R_d+Y) \geq \\
    \geq k+(k-t+d)-(k+d-1)\geq k-t+1.
  \end{multline}
  %since $d\leq t$;
  Since $R_i\subseteq C_{i_1\dots i_t}$ for any $i$, it follows
  \begin{equation}
    \dim (C_{i_1\dots i_t}\cap Y)\geq
    \dim (Y\cap (R_1+\dots+R_t)) \geq k-t+1,
  \end{equation}
  which is a contradiction because $\dim(C_{i_1\dots i_t}\cap Y)=k-t$, since $Y\in\cC_t(n,k)$ (see Proposition~\ref{equiv}).

  So, in order to get the thesis,  we need to prove Claim~\ref{claim}.

  Let $S$ be a subspace of $Y/(H\cap Y)$ with $\dim(S)=s.$ We show by induction on $s$ that $S$ is colorable. First, by our hypotheses,
  for any $p\in S$ we have $\varphi_H(p)\neq\infty$.

%Observe first that $\dim(Y/(X\cap Y))=d$.

   Suppose $\dim(S)=2.$ The projective space $\PG(S)$ is then a projective line of $\PG(Y/H\cap Y))$ so there are $|\FF|+1$ points in
         $\PG(S).$ By hypothesis we have $|\FF|\geq {n\choose t}$
         possible colors  (see Definition~\ref{psi-def}). Hence
         there are at least $2$ linearly
         independent vectors $p_1$ and $p_2$ in $S$ such that $\psi_H(p_1)=\psi_H(p_2).$ Hence, ${\mathfrak F}(S)\neq\emptyset$.
     %    the set of monochromatic bases
     %  of $S$ is non-empty.
         By Lemma~\ref{equiv colorable}, $S$ is colorable.

   Suppose now $\dim(S)>2.$ Put $s:=\dim(S).$ By induction, all subspaces $S'$ of $S$ with dimension $\dim(S')=\dim(S)-1$ are colorable,
         that is they all admit a monochromatic basis.
         For any $(s-2)$-subspace $S''$ of $S$  there are
    $|\FF|+1$ distinct  $(s-1)$-dimensional subspaces $S'$ of $S$ with
    $S''\leq S'\leq S$. Also $\PG(S/S'')$ is a projective line.
    Since $|\FF|+1>{n\choose t}$, there are at least two of such subspaces,
    say $S_1$ and $S_2$ with $S_1\neq S_2$ (hence $\langle S_1,S_2\rangle=S$) which have the same
    color $\psi_H(S_1)=\psi_H(S_2)$.%; their sum has dimension $\dim(S_1+S_2)=(s-1)+(s-1)-(s-2)=s$.

    Let $B_1$ and $B_2$ be bases of respectively $S_1$ and $S_2$
    with $\psi_H(B_1)=\psi_H(B_2)$ (=$\psi_H(S_1)$).
    By Lemma~\ref{trivial}, there is a basis $B$ of $S$ contained
    in $B_1\cup B_2$. So $S$ admits at least one monochromatic basis and
    ${\mathfrak F}(S)\neq\emptyset$. By Lemma~\ref{equiv colorable}, $S$ is colorable.
     \par
     Hence,  it is always possible   to determine a monochromatic set of $d$ independent vectors $\{p_1,p_2,\dots,p_d\}$ of $Y$ such that $[p_1],\dots,[p_d]$ are independent
    in $Y/(X\cap Y).$
%    with color
%$\psi_H(p_1)=\psi_H(p_2)=\dots=\psi_H(p_d)=(i_1,i_2,\dots,i_t)\in{\{1,\dots,n\}\choose t}$.
So, we have (recall that $H_i=\langle H,p_i\rangle$)
    \begin{equation}\label{Ri}
      \dim(H_1\cap C_{i_1\dots i_t})=
      \dim(H_2\cap C_{i_1\dots i_t})=\dots =
      \dim(H_d\cap C_{i_1\dots i_t})= k-t+1.
    \end{equation}
    Since $H\subseteq X$ and $X\in\cC_t(n,k)$, we have $\dim(H\cap C_{i_1\dots i_t})\leq k-t$.
    Recalling  that by definition $R_i:=H_i\cap C_{i_1\dots i_t}=\langle H,p_i\rangle\cap C_{i_1\dots i_t}$, we have, by~\eqref{Ri}, that
 $\dim(R_i)=\dim(H\cap C_{i_1\dots i_t})+1=k-t+1$ and $\dim(H\cap C_{i_1\dots i_t})=k-t.$

    In particular, (note that $p_i\not\in H$), it is always possible to find for $1\leq i\leq d$  an element
    $h_i\in H$ and a non-null element $\alpha_i\in\FF_q$ such that the point $\alpha_i p_i+h_i\in R_i$ is such that
    \[ \alpha_i p_i+h_i \in C_{i_1\dots i_t}; \]
up to a scalar multiple we can assume $\alpha_i=1$ for all $i$.

    We now show that $\dim(R_1+R_2+\dots+R_d)\geq k-t+d$.
    Suppose the contrary. Then, without loss of generality, we can
    assume
    $R_d\subseteq R_1+R_2+\dots+R_{d-1}$.
    In particular
    \[  p_d+h_d\in R_1+\dots+R_{d-1}, \]
    whence
    \[ p_d = \beta_1 p_1+\dots+\beta_{d-1} p_{d-1}+h \]
    with $h\in H$ a suitable element and $\beta_i\in \FF$
    for $1\leq i\leq d-1.$
    So, given that $p_1,\dots,p_d\in Y$,
    \[ p_d-(\beta_1 p_1+\dots+\beta_{d-1}p_{d-1})=h\in H\cap Y=X\cap Y, \]
    that is
    \[ [p_d]+[p_1]+\dots+[p_{d-1}]=[0] \]
    in $Y/(X\cap Y).$
    This contradicts the first part (already proved) of Claim~\ref{claim}, since $[p_1],\dots, [p_d]$ are linearly independent vectors
    of $Y/(X\cap Y)$.
    It follows $\dim(R_1+R_2+\dots+R_d)\geq k-t+d$.
    So Claim~\ref{claim} holds. This  completes the proof of the theorem.
  \end{proof}
  Note that if $\FF:=\FF_q$ is a finite field, then Lemma~\ref{ml7} gives the following
  \begin{corollary}
  \label{lcmm1}
  Let $X,Y\in\cC_t(n,k)$ with $\dim(X\cap Y)=k-d\geq k-t.$
If $\FF:=\FF_q$ and $q\geq{n\choose t}$ then there exist $[d]_q$ distinct codes  $Z\in\cC_t(n,k)$ such that $\dim(X\cap Z)=k-1$ and $\dim(Z\cap Y)=k-d+1$.
\end{corollary}
\begin{proof}
For any hyperplane $H$ of $X$ containing $X\cap Y$ it is possible to apply the argument in the proof of Lemma~\ref{ml7}. Thus there are at least $[d]_q$ distinct codes  $Z$ with the required property.
\end{proof}
We point out that for $d=2$, Corollary~\ref{lcmm1} states the same
as \cite[Lemma 1]{KPP18}. The following extends \cite[Lemma 2]{KPP18}
to the case $d\geq 2$ as well as to when $\FF$ is infinite.

  \begin{lemma}
  \label{d:>}
  Suppose $\FF$ is  a field with $|\FF|\geq {n\choose t}$. Then for any $X\in\cC_t(n,k)$ and for
  every $U\subset X$ with $\dim(U)<k-t$ there exists $X'\in\cC_t(n,k-1)$
  satisfying $U\subset X'\subset X$.
\end{lemma}
\begin{proof}
A hyperplane $H$ of $X$ is an element of
  $\cC_t(n,k-1)$ if and only if $H$ does not contain
  $X\cap C_{i_1\dots i_t}$ for any $i_1<\dots<i_t$.
  Indeed, if $C_{i_1\dots i_t}\cap X\subseteq H$ for some $i_1<\dots<i_t$, then
 $\dim(H\cap C_{i_1\dots i_t})\geq\dim(X\cap C_{i_1\dots i_t})= k-t>k-t-1$ and $H\not\in\cC_t(n,k-1)$.
 Conversely, suppose that for any $i_1<\dots<i_t$, $C_{i_1\dots i_t}\cap X\not\subseteq H$;
 then $\dim(H\cap X\cap C_{i_1\dots i_t})=\dim(H\cap C_{i_1\dots i_t})=k-1-t$ for any
 choice of the indexes; so $H\in\cC_t(n,k-1)$.

By Definition of $C_{i_1\dots i_t}$ (see Definition~\ref{coord subspace}), there exist at most ${n\choose t}$ distinct spaces $C_{i_1\dots i_t}$.
 Now we distinguish two cases.
 \begin{itemize}
 \item
   If $\FF:=\FF_q$ is a finite field,
   each of the spaces $C_{i_1\dots i_t}$ is
   contained in  $[t]_q$ distinct hyperplanes of $X$. So
   the number of hyperplanes containing at least one $X\cap C_{i_1\dots i_t}$ is
   at most
   \[ {n\choose t}[t]_q. \]
   On the other hand $U$ is contained in $[m]_q$ distinct hyperplanes where
   $m=\dim(X/U)$. Since $m>t$ we have
   \[ [m]_q\geq[t+1]_q=q^{t}+q^{t-1}+\dots+q+1\geq
     {n\choose t}q^{t-1}+{n\choose t}q^{t-2}+\dots+{n\choose t}+1>{n\choose t}[t]_q. \]
   This shows that there is at least one hyperplane $X'$ of $X$ containing $U$
   and none of the $C_{i_1\dots i_t}$.
   \item
Suppose $\FF$ is an infinite field and denote by $X^*$ the dual space of $X$.
     Then, for every $U\subseteq X$
     with $\dim U<k-t$ the set of hyperplanes $X'$ of $X$ containing $U$
     determine a subspace $\mathcal U$ of $\PG(X^*)$
     of vector dimension at least $k-(k-t-1)=t+1$.
     Since each of the spaces $X\cap C_{i_1\dots i_t}$ has dimension
     $k-t$ (because $X\in \cC_t(n,k)$), the set
     ${\mathcal H}_{i_1\dots i_t}$ of hyperplanes containing
     $C_{i_1\dots i_t}$ is a subspace of $\PG(X^*)$ of vector
     dimension $t$.
     In particular, the set of all hyperplanes of $X$ containing
     at least one $C_{i_1\dots i_t}$ is
     the union of ${n\choose t}$ subspaces of $\PG(X^*)$ each
     of vector dimension $t$.

     Since the field $\FF$ is infinite, it is impossible
     for a projective space of vector dimension  at least $t+1$ to be
     the union of a finite number of projective spaces of dimension
     $t$.

     It follows that there is at least one element
     \[ X'\in{\mathcal U}\setminus\bigcup_{\begin{subarray}{c}
           (i_1\dots i_t)\in \\
         {{\{i_1,\dots,i_n\}}\choose t}\end{subarray}}{\mathcal H}_{i_1\dots i_t}. \]
       This leads to the same conclusion as in the case
       in which $\FF$ is finite.
  \end{itemize}
\end{proof}

We are now ready to prove the following theorem, which extends~\cite[Theorem 1]{KPP18} to arbitrary $t.$
\begin{theorem}
\label{m3}
Suppose $\FF$ is  a field with $|\FF|\geq {n\choose t}$. Then $\Delta_t(n,k)$ is connected and  isometrically embedded
into $\Gamma(n,k)$.%; that is to say
%that $d_t(X,Y)=d(X,Y)$ for any $X,Y\in \cC_t(n,k)$ and
%$\Delta_t(n,k)$ is connected.
\end{theorem}
\begin{proof}
  Take $X,Y\in\cC_t(n,k)$ with $X\not= Y$. Put $\dim(X\cap Y):=k-d$. If $k-d\geq k-t$, then the thesis follows from Lemma~\ref{ml7}.
  Suppose now $k-d<k-t$. By Lemma~\ref{d:>} there exists $X'\subseteq X$
  with $X\cap Y\subseteq X'$ and $X'\in\cC_t(n,k-1)$.
Let $Y'$ a $(k-d)+1$- dimensional subspace of $Y$ containing
  $X\cap Y$. Put $T=\langle X',Y'\rangle$. Then $\dim(T)=k$; also
  \[ \dim(T\cap C_{i_1\dots i_t})\leq\dim(X'\cap C_{i_1\dots i_t})+1=
    k-1-t+1=k-t \]
  for all $1\leq i_1<i_2<\dots<i_t\leq n$.
  In particular $T\in\cC_t(n,k)$ and $\dim(X\cap T)=k-1$, $\dim(Y\cap T)=k-d+1$.
  By recursively applying this argument we get that $\Delta_t(n,k)$
  is connected.

  By construction, the length $d_t(X,Y)$ of the path joining $X$ and $Y$
  in $\Delta_t(n,k)$ is at most $k-\dim(X\cap Y):=d(X,Y)$, i.e. $d_t(X,Y)\leq d(X,Y)$.
  Since $d_t(X,Y)\geq d(X,Y)$ in general, we get the thesis.
\end{proof}
\begin{remark}
  As a consequence of Theorem~\ref{m3},
  \[ \mathrm{diam}(\Delta_t(n,k))\leq\mathrm{diam}(\Gamma(n,k)) \]
  when $|\FF|\geq{n\choose t}$. In the following section we
  shall show that these two diameters are actually the same.
\end{remark}

 \subsection{Codes at maximum distance}
\label{maxd}
In this section we do not assume any hypothesis on the parameters,
apart that they have been chosen so that  there exists at least one $[n,k]$-linear code with dual minimum distance at least $t,$ i.e. $\cC_t(n,k)\neq\emptyset$.
Hence, the graph $\Delta_t(n,k)$ is not assumed to be connected. This observation justifies the following.

\begin{definition}
  We say that two codes in $X,Y\in\cC_t(n,k)$ are \emph{opposite} in $\Delta_t(n,k)$
  if they belong to the same connected component
  $\Delta_t^X(n,k)=\Delta_t^Y(n,k)$
  of $\Delta_t(n,k)$ and $d_t(X,Y)=\mathrm{diam}(\Delta_t^X(n,k))$.
\end{definition}
\begin{definition}
  We say that two $k$-dimensional subspaces $X,Y\subseteq V$
	are \emph{opposite}
  in $\Gamma(n,k)$ if  $\dim(X\cap Y)=\max\{2k-n,0\}$.
\end{definition}
 Observe that if $2k\leq n$, being opposite in $\Gamma(n,k)$ means
 $\dim(X+Y)=2k$ (equivalently,
 $X\cap Y=\{0\}$), while if $2k>n$,  it means $\dim(X+Y)=n.$

  \begin{lemma}
   \label{opGamma}
  Suppose $t\leq k\leq n$, $\cC_t(n,k)\neq\emptyset$ and $|\FF|>\max\{k,n-k\}+1$. Then, for
  any code $C\in\cC_t(n,k)$ there exists a code $D\in\cC_t(n,k)$
  which is equivalent and opposite to $C$ in $\Gamma(n,k)$.
\end{lemma}
\begin{proof}
  Suppose $2k\leq n$ and let $C\in \cC_t(n,k)$ with $G$ as generator matrix. By elementary row operations on $G$, which leave $C$ invariant,
  and column operations by means of
  $\rho\in\cM(V)$ (see Definition~\ref{def equiv}), we can obtain a generator matrix
	\[ G':=\begin{pmatrix}
      I & A & B
    \end{pmatrix} \]
  for an equivalent code $\rho(C)=:C'\in \cC_t(n,k)$.
  Here $I$ is the $k\times k$ identity matrix,
  $A$ is a $k\times k$ matrix of rank $t$ (since any $t$ columns of a generator matrix of a code in $\cC_t(n,k)$ are linearly independent)
  and $B$ is a $k\times (n-2k).$ Take $\lambda\in \FF\setminus\{0\}$ and consider the matrix
     \[ G_{\lambda}'':=\begin{pmatrix}
    \lambda A & I & B
  \end{pmatrix}. \]
	
Since $G_{\lambda}''$ is obtained from $G'$ by applying transformations
induced by the monomial group $\cM(V)$,
the code $C_{\lambda}''$ having $G_{\lambda}''$ as generator matrix, is equivalent to
$C'$. In particular $C_{\lambda}''\in \cC_t(n,k).$

We want to show that it is always possible to choose $\lambda$
so that $C_{\lambda}''$ and $C'$ are in direct sum as subspaces of $V$,
that is the matrix
\[ \begin{pmatrix}
    I & A & B \\
    \lambda A & I & B
  \end{pmatrix} \]
has rank $2k$.
By elementary row operations,
subtracting from the second block $\lambda A\begin{pmatrix} I&A&B\end{pmatrix}$,
we see that the rank of the matrix above
is the same as the rank of
\[ \begin{pmatrix}
    I & A & B \\
    0 & I-\lambda A^2 & (I-\lambda A)B
  \end{pmatrix}. \]
In particular, this rank is definitely $2k$ if
$\det(I-\lambda A^2)\neq 0$.

On the other hand $\det(I-\lambda A^2)=0$ if and only if
$\lambda^{-1}$ is an eigenvalue of $A^2$. Since $A^2$ is a $k\times k$ matrix,
 of rank at most $t$,
the number of its non-null eigenvalues
is at most $t\leq k<|\FF^*|$. So, there is at least one
$\lambda\in\FF^*$ such that $\lambda^{-1}$ is not a non-null eigenvalue of $A^2.$ For such a $\lambda$, the matrix $G_{\lambda}''$ represents a code
$C_{\lambda}''\in\cC_t(n,k)$ such that $\dim(C'\cap C_{\lambda}'')=0$, that is
$C'$ and $C_{\lambda}''$ are opposite in $\Gamma(n,k)$.
\par

Suppose now that $2k>n$. Let $C\in\cC_t(n,k)$ be a code with
generator matrix $G$. Using elementary row
operations on $G$ and a monomial transformation $\rho\in\cM(V)$, we
can obtain a code $C':=\rho(C)$ equivalent to $C$ whose generator matrix
$G'$ is in systematic form, i.e.
\[ G'=\begin{pmatrix} I_{n-k} & 0 & A_1 \\
    0    & I_{2k-n} & A_2
  \end{pmatrix}, \]
where $A_1$ and $A_2$ are suitable matrices of dimensions
respectively $(n-k)\times (n-k)$ and $(2k-n)\times (n-k)$.
For $\lambda\in \FF\setminus \{0\}$, let now $G''_{\lambda}=\begin{pmatrix} \lambda A_1 & 0 & I_{n-k} \\
  \lambda A_2 & I_{2k-n} & 0
\end{pmatrix}$ and $C''_{\lambda}$ be the
code with generator matrix $G''_{\lambda}$.
The matrix $G''_{\lambda}$ is obtained by permuting and multiplying some of the
columns of the matrix $G$ by a non-zero scalar $\lambda$;
as such the code $C''_{\lambda}$ generated by $G''_{\lambda}$ is equivalent to $C$ and $C'$; thus,
$C''_{\lambda}\in\cC_t(n,k)$ for all $\lambda\neq0$.

Since $2k>n$, the codes $C'$ and $C_{\lambda}''$ are opposite if and only if
$\dim(C'+C''_{\lambda})=n$, that is to say the rank of the matrix
$\bar{G}_{\lambda}=\begin{pmatrix} G' \\ G''_{\lambda} \end{pmatrix}$ is maximum and equal to $n$.

Explicitly, the structure of the matrix $\bar{G}_{\lambda}$ is
  \[ \bar{G}_{\lambda}=
    \begin{pmatrix}
      I_{n-k} & 0 & A_1 \\
      0     & I_{2k-n} & A_2 \\
    \lambda  A_1   &    0    & I_{n-k}  \\
    \lambda  A_2   & I_{2k-n} & 0
    \end{pmatrix}.
  \]
  By using column operations we see that
\begin{multline*} \rank
   \begin{pmatrix}
      I_{n-k} & 0 & A_1 \\
      0     & I_{2k-n} & A_2 \\
    \lambda  A_1   &    0    & I_{n-k}  \\
    \lambda  A_2   & I_{2k-n} & 0
    \end{pmatrix}\geq
   \rank
   \begin{pmatrix}
      I_{n-k} & 0 & A_1 \\
      0     & I_{2k-n} & A_2 \\
    \lambda  A_1   &    0    & I_{n-k}  \\
    \end{pmatrix}= \\ \rank
    \begin{pmatrix}
      I_{n-k} & 0 & A_1 \\
      0     & I_{2k-n} & 0 \\
      \lambda A_1   &    0    & I_{n-k}  \\
    \end{pmatrix}=
    \rank
    \begin{pmatrix}
      I_{n-k} & 0 & 0 \\
      0     & I_{2k-n} & 0 \\
      \lambda A_1   &    0    & I_{n-k}-\lambda A_1^2  \\
    \end{pmatrix}.
      \end{multline*}
      So, if $\lambda^{-1}$ is not an eigenvalue of $A_1^2$,
      we have that $\rank(\bar{G}_{\lambda})=n$.
      As the matrix $A_1^2$ has dimension $(n-k)\times(n-k)$,
      we have that $A_1^2$ has at most $n-k$  eigenvalues;
      as $|\FF^*|>n-k$ there are some values of $\lambda\neq 0$ such that
      this rank is maximum; for these values of $\lambda$ we get
      that $C',C''_{\lambda}\in\cC_t(n,k)$ are opposite.
      \par
      Observe now that for any two codes $X,Y\in\cC_t(n,k)$ and any $\eta \in \cM(V)$, we have
      $d(X,Y)=d(\eta(X),\eta(Y))$ in $\Gamma(n,k)$.
      Since  $C'=\rho(C)$ for $\rho\in \cM(V)$, put $D=\rho^{-1}(C_{\lambda}''),$ where $C_{\lambda}''$ is the code constructed above. Then,
      \[ d(C,D)=d(\rho^{-1}(C'),\rho^{-1}(C_{\lambda}''))=
        d(C',C_{\lambda}'').
      \]
      It follows that $C$ and $D$ are codes
      in $\cC_t(n,k)$ which are equivalent and opposite in $\Gamma(n,k)$.
    \end{proof}
    \begin{corollary}
      \label{diam}
      Suppose $\cC_t(n,k)\neq\emptyset$ and
      $\Delta_t(n,k)$ to be isometrically embedded into
      $\Gamma(n,k)$.
      If $|\FF|>\max\{k,n-k\}+1$, then
      \[ \mathrm{diam}(\Delta_t(n,k))=
      \mathrm{diam}(\Gamma(n,k)). \]
  \end{corollary}
  \begin{proof}
    By Lemma~\ref{opGamma}, there are at least two codes
    $X,Y\in\cC_t(n,k)$ with $d(X,Y)=\mathrm{diam}(\Gamma(n,k))$. Since $\Delta_t(n,k)$ is isometrically
    embedded in $\Gamma(n,k)$ we also have
    $d_t(X,Y)=\mathrm{diam}(\Gamma(n,k))$. On the other hand,
    for any $X,Y\in\cC_t(n,k)$,
    \[ d_t(X,Y)=d(X,Y)\leq\mathrm{diam}(\Gamma(n,k)). \]
    It follows that $\mathrm{diam}(\Delta_t(n,k))=\mathrm{diam}(\Gamma(n,k))$.
  \end{proof}
  Note that for $q\geq{n\choose t}$, the assumptions of
  Corollary~\ref{diam} hold.
%\subsection{Proof of Theorem~\ref{c1}}
%\label{pfmain}
%We show that $\cC(n,k)$ is non-empty under some condition on $t,k,n.$

Theorem~\ref{c1} follows from Lemma~\ref{mds}, Theorem~\ref{m3} and  Corollary~\ref{diam}.

\section*{Acknowledgments}
The first and second authors are affiliated with GNSAGA of INdAM (Italy) whose support they
acknowledge. The third author was supported by the National Science Centre, Poland
(grant number 2019/03/X/ST/00944).

% \vskip.2cm\noindent
% \begin{minipage}[t]{\textwidth}
% \hskip-1cm Authors' addresses:
% \vskip.2cm\noindent\nobreak
% \centerline{
%   \hskip-1cm
% \begin{minipage}[t]{9cm}
% Ilaria Cardinali \\
% Department of Information Engineering and Mathematics\\
% University of Siena\\
% Via Roma 56, I-53100, Siena, Italy\\
% ilaria.cardinali@unisi.it\\
% \end{minipage}\hfill
% \begin{minipage}[t]{6.3cm}
% Luca Giuzzi\\
% D.I.C.A.T.A.M. (Section of Mathematics) \\
% Universit\`a di Brescia\\
% Via Branze 53, I-25123, Brescia, Italy \\
% luca.giuzzi@unibs.it
% \end{minipage}}
% \centerline{
% \begin{minipage}[t]{9cm}
%   Mariusz Kwiatkowski \\
%   Department of Logic and Foundations of
%   Computer Science \\
%   University of Warmia and Mazury in Olsztyn \\
%   Sloneczna 54 Street, PL-10710 Olsztyn, Poland \\
%   mkw@matman.uwm.edu.pl \\
%   \end{minipage}}
% \end{minipage}

%\bigskip


\begin{thebibliography}{99}
    \bibitem{DRG} A.E. Brouwer, A.M. Cohen, A. Neumaier,
      \emph{Distance-Regular Graphs},
      Springer Verlag (1989).
  \bibitem{KP16} M. Kwiatkowski, M. Pankov,
    \emph{On the distance between linear codes},
    Finite Fields Appl. {\bfseries{39}} (2016), 251--263.
  \bibitem{KPP18} M. Kwiatkowski, M. Pankov, A. Pasini,
    \emph{The graphs of projective codes}
    Finite Fields Appl. {\bfseries{54}} (2018), 15--29.
  % \bibitem{KP19} M. Kwiatkowski, M. Pankov,
  %   \emph{Graphs related to $2$-dimensional simplex codes},
  %   preprint.
  \bibitem{MS} F.J. MacWilliams, N.J.A. Sloane,
    \emph{The Theory of Error-Correcting Codes},
    North-Holland (1983).
  \bibitem{S} E.E. Shult,
    \emph{Points and Lines},
    Springer-Verlag (2011)
  \end{thebibliography}
 \end{document}